\documentclass[reqno,keywordsasfootnote,10pt]{article}
\usepackage[english]{babel}
\usepackage[latin1]{inputenc}
\usepackage{amsmath,amsthm,amsfonts,amssymb,times,epsfig}
\numberwithin{equation}{section}

  \usepackage{epstopdf}
\newtheorem{thm}{Theorem}[section]
\newtheorem{lem}[thm]{Lemma}
\newtheorem{prop}[thm]{Proposition}
\newtheorem{cor}[thm]{Corollary}
\newtheorem{df}[thm]{Definition}

\newcommand{\E}{\mathbb{E}}

\newcommand{\prob}{\mathbb{P}}

\newcommand{\R}{\mathbb{R}}
\newcommand{\Z}{\mathbb{Z}}
\newcommand{\N}{\mathbb{N}}

\newcommand{\distrib}{\overset{\mbox{$\mathcal D$}}{=}}
\newcommand{\hilbert}{\mathcal{H}}

\def\be{\begin{equation} }
\def\ee{\end{equation} }
\begin{document}
\date{September 3, 2008}
\title{A Remark on the Infinite-Volume Gibbs Measures of Spin Glasses}
\author{\Large Louis-Pierre Arguin \\[2ex] 
Weierstrass Institute for Applied Analysis and Stochastics
\\Berlin, Germany
}  

\maketitle  
\begin{abstract}
In this note, we point out that infinite-volume Gibbs measures of spin glass models on the hypercube can be identified as random probability measures on the unit ball of a Hilbert space. This simple observation follows from a result of Dovbysh and Sudakov on weakly exchangeable random matrices. 
Limiting Gibbs measures can then be studied as single well-defined objects. 
This approach naturally extends the space of Random Overlap Structures as defined by Aizenman, Sims and Starr. We discuss the Ruelle Probability Cascades and the stochastic stability within this framework. As an application, we use an idea of Parisi and Talagrand to prove that if a sequence of finite-volume Gibbs measures satisfies the Ghirlanda-Guerra identities, then the infinite-volume measure must be singular as a measure on a Hilbert space.
\end{abstract}

\section{The Gibbs Measure of Spin Glasses}
A spin glass on $N$ spins is a collection of Gaussian variables or Hamiltonians $$(H_N(\sigma),\sigma\in\{-1,1\}^N)$$ whose covariance is function of the overlap between configurations. For simplicity, we restrict ourselves to mean-field systems defined on the hypercube though the framework proposed here has a more general scope. For example, the Sherrington-Kirkpatrick (SK) model corresponds to the case where the Hamiltonians have covariance $\E[H_N(\sigma)H_N(\sigma')]=N\ q^2_{\sigma\sigma'}$ where the overlap $q_{\sigma\sigma'}=\frac{1}{N}\sum_{i=1}\sigma_i\sigma_i'$ is related to the Hamming distance on the hypercube by $\#\{i:\sigma_i\neq \sigma_i'\}= \frac{N}{2}(1-q_{\sigma\sigma'})$.  A goal is to understand the Gibbs measure $\mu_{\beta,N}$ of the model for $\beta\geq 0$ in the limit $N\to\infty$ where
\be
\mu_{\beta,N}:=\left(\frac{e^{-\beta H_N(\sigma)}}{\sum_{\sigma'}e^{-\beta H_N(\sigma')}}, \ \sigma\in\{-1,+1\}^N\right)\ .
\label{eqn:gibbs}
\ee

The convergence of the Gibbs measure is usually understood as the convergence of the appropriate observables of the model. However, it is not clear {\it a priori} that the collection of limiting observables are functionals of a non-trivial single object amenable itself to analysis.
This is to be compared with the classical spin systems on the lattice $\Z^d$. In this case, the limiting measure is well-defined as a probability measure on $\{-1,+1\}^{\Z^d}$ equipped with the $\sigma$-algebra generated by functions supported on a finite box. And limiting Gibbs measures, that are also translation-invariant, are characterized as optimizers of the Gibbs variational principle for the free energy \cite{Simon}.

It would be desirable to obtain such a transparent framework for mean-field spin glasses. A definition for the infinite-volume measure which is well-suited when the Hamiltonians exhibit a hierarchical structure is proposed in \cite{Bovier}. In the general case, Aizenman, Sims and Starr introduced the notion of {\it Random Overlap Structures} or ROSt's \cite{AS22}. A ROSt is a random pair $(\xi,Q)$ where $\xi=(\xi_i,i\in\N)$ is a collection of summable weights and $Q=\{q_{ij}\}$ is a symmetric positive semi-definite matrix with $q_{ii}=1$. For instance, the Gibbs measure \eqref{eqn:gibbs} is a ROSt. They prove a variational principle for the free energy of the SK model over the space of ROSt's. However, the current definition does not specify the topology of the space and therefore the appropriate definition of convergence of the measure. Moreover, as pointed out by the authors in \cite{AS22}, one should not rule out the possibility that the limiting ROSt of a model admits continuous weight $\xi$. Thus the space of ROSt's should be enlarged to include such cases.

The purpose of this communication is to observe that, for mean-field spin glasses, the collection of observables in the thermodynamic limit defines a random probability measure on the unit ball of a separable Hilbert space $\hilbert$. The space $\mathcal{M}$ of such elements is compact and Polish when endowed with the topology coming from the sampling of replicas. The space $\mathcal{M}$ becomes a natural extension of the space of ROSt's, the latter corresponding to those measures supported on a countable set of vectors of norm $1$.

The idea is the following. Observables on mean-field spin glasses are obtained by sampling configurations with the measure \eqref{eqn:gibbs}, taking overlaps and averaging over the randomness e.g. for some integers $k_1,...,k_{s(s-1)/2}$
$$\overline{\mu_{\beta,N}^{(s)}}(q_{12}^{k_1}...q_{s-1,s}^{k_{s(s-1)/2}}) $$
where we denote the product measure on $s$ replicas by $\mu^{(s)}$, the overlap between the $i$-th and the $j$-th sampled configurations by $q_{ij}$ and the average over the randomness of the Hamiltonians by $\overline{\mu}$. Since $q_{\sigma\sigma}=1$ and the overlap matrix is positive definite, the measure \eqref{eqn:gibbs} can be thought of as a random probability measure on the unit sphere of $\hilbert$ by simply embedding the hypercube isometrically in $\hilbert$. The overlap between two configurations is then the inner product between the vectors and the randomness is the one induced by the random Hamiltonian. Since the overlaps take value in $[-1,1]$, compactness ensures the existence of a subsequence of $\mu_{\beta,N}$ for which every observable of the above sort converges \cite{AC}. It turns out that these limiting observables come also from the sampling of a single random measure when measures on the whole unit ball are considered - thanks to a result of Dovbysh and Sudakov on exchangeable matrices \cite{DS} (c.f. Theorem \ref{thm:DS} and Proposition \ref{prop weak limit}).

The interest of the above observation is the study of the infinite-volume Gibbs measure as a single well-defined object, a random probability measure on $\hilbert$. This perspective raises new tantalizing questions besides the well-known conjecture on the ultrametricity of the overlaps.
For example, do the limiting Gibbs measures admit a continuous part ? If they are actually singular, are they pure point ? Are the measures supported on a non-random sphere ? Etc.

The structure of the paper is as follows. The details of the construction of the space $\mathcal{M}$  are presented in Section \ref{section we}. We also explain how the Gibbs measure at high temperature and the important examples of the Ruelle Probability Cascades fit naturally in this framework c.f. Proposition \ref{prop:rpc}. Some properties of the limiting measure are studied in Section \ref{section gg}. We start by discussing how the so-called stochastic stability property of the Gibbs measure, as studied in \cite{AC,Contucci, GG, Guerra_cavity}, is expressed by the infinite-volume Gibbs measure. In particular, we look at the consequences of the Ghirlanda-Guerra identities on the limiting object. Following the idea of Parisi and Talagrand \cite{Parisi}, we show that measures satisfying them must be singular in a suitable sense c.f. Corollary \ref{cor:gg}.

\section{Gibbs Measure and Weakly Exchangeable Matrix}
\label{section we}
\subsection{Definitions}
In this section, we first define the space of random overlap matrices that are weakly exchangeable. We then present the theorem of Dovbysh and Sudakov that characterizes such matrices in terms of random probability measures on a separable Hilbert space. The existence of infinite-volume Gibbs measures is a straightforward consequence of the latter.

The set of positive semi-definite symmetric $\N\times\N$ matrices with $1$ on the diagonal is plainly a compact Polish space when considered as a closed subset of $[-1,1]^{\N\times\N}$ equipped with the product topology. We call an element of this space an {\it overlap matrix}. A {\it random overlap matrix} is a Borel probability measure on the space of overlap matrices. This space is also compact and Polish when equipped with the weak-$*$ topology induced by the continuous functions on overlap matrices. Any continuous functional in this topology can be approximated by linear combinations of monomials of the form
\be
\E\left[\prod_{1\leq i<j\leq s}q_{ij}^{k_{ij}}\right]
\label{eqn:fct}
\ee
for some random overlap matrix $Q=\{q_{ij}\}$ of law $\prob$, $s\in\N$ and $k_{ij}\in\N$. 
\begin{df} 
A random overlap matrix $Q$ is said to be weakly exchangeable if for any permutation matrix $\tau$ on finite elements
$$ \tau \ Q \ \tau^{-1}\ \distrib Q\ .$$
We denote the space of weakly exchangeable random overlap matrices by $\mathcal{M}$. 
\end{df}
The space $\mathcal{M}$ is a compact convex subset of the space of random overlap matrices. By Choquet's theorem, any element can be written as the barycenter of the extreme elements $\text{Ext} \ \mathcal{M}$.
The characterization of the extremes, and hence of the whole set, was achieved by Dovbysh and Sudakov \cite{DS}.

\begin{thm}[Dovbysh-Sudakov]
Let $\hilbert$ be an infinite-dimensional separable Hilbert space. $\text{Ext} \ \mathcal{M}$ is in bijection with the set of probability measures on the closed unit ball of $\hilbert$, considered up to isometry. 
\label{thm:DS}
\end{thm}
The correspondence established by Dovbysh and Sudakov works  as follows. The law of the non-diagonal entries of $Q\in \text{Ext} \ \mathcal{M}$ is the law of the Gram matrix constructed from iid $\mu$-sampled vectors where $\mu$ is a probability measure on the unit ball of $\hilbert$. Precisely, for $i\neq j$
$$ q_{ij}\distrib (v_i,v_j)_{\hilbert}$$
where $(v_i,i\in\N)$ is a sequence of iid $\mu$-distributed vectors. We stress that, even though $(v_i,v_i)_\hilbert$ can be strictly smaller than $1$, $q_{ii}=1$ for all $i$. It is clear that two measures $\mu$ and $\mu'$ which differ by an isometry of $\hilbert$, $\mu'=\mu\circ U^{-1}$ for some isometry $U$, yield the same law. Conversely, if no such isometry exists, they produce two different elements of $\mathcal{M}$. A simple element of $\text{Ext}\ \mathcal{M}$ is the identity matrix $Q=\text{Id}$ a.s. In this case, the corresponding $\mu$ is the delta measure on the $0$ vector. 

From this perspective, since $\mathcal{M}$ is the closed convex hull of $\text{Ext} \ \mathcal{M}$, the whole set is in correspondence with the convex combinations of probability measures, or random probability measures, on the unit ball of $\hilbert$. We call the random measure associated with the weakly exchangeable matrix $Q$ {\it the directing measure of $Q$}. Similarly as above, we write $\mu^{(s)}$ for the product measure of $s$ copies of $\mu$, $\mu^{(s)}:=\mu\times ...\times \mu$, and $\overline{\mu^{(s)}}$ when it is averaged over the randomness of $\mu$.

From now on, we identify $\mathcal{M}$ with the set of random probability measure on the unit ball of $\hilbert$, considered up to isometry. The topology on $\mathcal{M}$ naturally translates into the topology given by sampled replicas or copies of the directing measure. Namely, if $Q$ is weakly exchangeable with directing measure $\mu$, \eqref{eqn:fct} becomes
\be
\overline{\mu^{(s)}}\left(\prod_{1\leq i<j\leq s}q_{ij}^{k_{ij}}\right)
\label{eqn:h}
\ee
The above considerations translate into a simple lemma.
\begin{lem}
The space $\mathcal{M}$ of random probability measure on $\hilbert$ considered up to isometry is a compact Polish space when equipped with the topology generated by the linear span of functions of the form \eqref{eqn:h}.
\end{lem}
We note that the space of ROSt's as described in \cite{AS22} is a subset of $\mathcal{M}$ corresponding to the directing measures that are supported on a countable number of vectors sitting on the unit sphere. We stress that, under the topology considered here, the norm of the vectors on which the directing $\mu$ is supported is not preserved under the convergence in $\mathcal{M}$. For example, the uniform measure on $N$ orthonormal vectors $e_i$, $\mu_N=\frac{1}{N}\sum_{i=1}^N\delta_{e_i}$, converges in $\mathcal{M}$ to the delta measure at $0$. 

Theorem \ref{thm:DS} ensures the existence of infinite-volume Gibbs measures of  mean-field spin glasses as the limit in $\mathcal{M}$ of finite Gibbs measures. 
\begin{prop}
Let $\mu_{\beta,N}$ be a Gibbs measure of the form \eqref{eqn:gibbs}. There exists a subsequence converging to some $\mu_{\beta}$ in $\mathcal{M}$. In particular, $\mu_\beta$ can be identified, up to isometry, to a random probability measure on the unit ball of $\hilbert$.            
\label{prop weak limit}
\end{prop}
\begin{proof}
As noted earlier, the measure $\mu_{\beta,N}$ of $\eqref{eqn:gibbs}$ can be seen as a random probability on $\hilbert$ by embedding the hypercube isometrically in $\hilbert$. The sequence of weakly exchangeable matrices built by sampling replicas with this measure has a converging subsequence by the compactness of $\mathcal{M}$. Moreover, the limiting random matrix must have a directing measure $\mu_\beta$ by Theorem \ref{thm:DS}. 
\end{proof}
Before investigating the properties of infinite-volume Gibbs measures as probability measures on $\hilbert$, we turn to some relevant examples.\\

\subsection{Examples}
\subsubsection{The Gibbs measure of the SK model at high temperature}
Let $\mu_{\beta,N}$ be the Gibbs measure of the SK model for some temperature $\beta$ and for zero magnetic field. It was proven in \cite{ALR} that for $\beta<1$ 
\be
\lim_{N\to\infty} \overline{\mu_{\beta,N}^{(2)}}(q_{12}=0)=1\ .
\label{eqn:ALR}
\ee
It follows that
\begin{prop}
The limit $\mu_\beta$ of any converging subsequence of $\{\mu_{\beta,N}\}$ in $\mathcal{M}$ is the delta measure on the zero vector $ \mu_\beta=\delta_{0}$ a.s.
\end{prop}
\begin{proof}
By \eqref{eqn:ALR} any limit of functionals of the form \eqref{eqn:h} is $0$. Therefore, these functionals evaluated at any limit point $\mu_\beta$ of the sequence must be also $0$. Hence the overlap matrix constructed from a sampling of $\mu_\beta$ is the identity almost surely and the claim follows.
\end{proof}
A similar result holds for the replica symmetric region with non-zero magnetic field. In that case, the limiting measure is supported on a single vector of norm $0<q<1$, where $q$ is the solution to the self-consistency equation. \\

\subsubsection{ The Ruelle Probability Cascades}
We detail how the Ruelle Probability Cascades as constructed by Ruelle \cite{Ruelle} and reconsidered by Bolthausen and Sznitman \cite{BS} conveniently fit in $\mathcal{M}$. 

We recall that a Bolthausen-Sznitman coalescent $\Gamma=(\Gamma(t),t\geq 0)$ is a continuous-time Markov process on the space of partitions of $\N$. A partition of $\N$ is a collection of disjoint blocks of integers whose union is $\N$. At time $t=0$, we set $\Gamma(0)=\N=\{\{1\},\{2\},...\}$. The partition at time $t+t'$ conditionally on the partition at time $t$ is obtained by lumping blocks of $\Gamma(t)$ (see Proposition 1.1 of \cite{BS} for the precise transition probabilities).  The lumping is such that $\Gamma(t)$ is an exchangeable partition for every $t$ i.e. the probability of $i$ belonging to a given block is independent of $i$. One defines the stopping time $\tau_{ij}$ as the infimum over the time $t$ such that $i$ and $j$ belong to the same block $\Gamma(t)$. Since the partitions are nested by definition, the $\tau$'s satisfy the ultrametric inequality $\tau_{ij}\leq \max \{\tau_{ik},\tau_{jk}\}$ for any $i,j,k\in\N$ c.f. Equation 0.11 in \cite{BS}. Moreover, by the exchangeability of the partitions, the joint distribution of the $\tau$'s is also invariant under a permutation of the indices.

From the Bolthausen and Sznitman perspective, a Ruelle Probability Cascade (RPC) is a time change of the coalescent. Namely, let $\Gamma$ be the Bolthausen-Sznitman coalescent and $x:[0,1]\to[0,1]$ a cumulative distribution function on $[0,1]$, one considers the random matrix $Q=\{q_{ij}\}$
\be
q_{ij}:=x^{-1}(e^{-\tau_{ij}})
\label{eqn:rpc}
\ee
for the right-continuous inverse $x^{-1}(e^{-t}):=\inf\{q\geq 0: x(q)> e^{-t}\}$. We claim that \eqref{eqn:rpc} defines an element of $\mathcal{M}$. Since $\Gamma(0)=\N$, $q_{ii}=1$ for all $i$. The matrix $Q=\{q_{ij}\}$ is clearly symmetric. It is also positive semi-definite: since $\tau_{ij}\leq\max\{\tau_{ik},\tau_{jk}\}$, then $q_{ij}\geq \min\{q_{ik},q_{jk}\}$ and $Q$ is the covariance matrix of the Gaussian field on the tree defined by the ultrametric $\tau$. Finally, $Q$ is weakly exchangeable from the exchangeability property it inherits from the $\tau$'s.
\begin{df}
A Ruelle Probability Cascade (RPC) $\mu_x$ with parameter $x:[0,1]\to [0,1]$ is the element of $\mathcal{M}$ constructed from the Bolthausen-Sznitman coalescent through \eqref{eqn:rpc}.
\end{df}
In the case where the function $x$ is a step function, the definition coincides with the construction of Ruelle as stated in Theorem 2.2 of \cite{BS}. In fact, the directing measure $\mu_x$ is exactly the cascade of orthonormal vectors with weight given by Poisson-Dirichlet processes constructed by Ruelle.
The advantage of the above definition is the fact that the continuous cases as well as the degenerate cases have a well-defined meaning. For example, the case $x(q)\equiv 1$ corresponds to the matrix $q_{ij}=1$ for $i,j\in\N$ with directing measure $\mu_x=\delta_e$ for some vector $e$ of norm $1$ and $x(q)=0$ for $q\in[0,1)$ yields $Q=\text{Id}$ with directing measure $\mu_x=\delta_0$. 

A good topology for RPC's turns out to be the $L^1([0,1])$-norm on the space of the parameters $x$. Indeed, the Parisi functionals seen as functionals on RPC's are continuous in this topology \cite{guerra_cont, lp_rpc}. It is interesting to note that this topology corresponds exactly to the topology the RPC's inherit as a subset of $\mathcal{M}$.
\begin{prop}
The $\mathcal{M}$ topology on the set of RPC's is equivalent to the $L^1([0,1])$ topology on the parameters $x:[0,1]\to[0,1]$. 

In other words, $\mu_{x_n}$ converges to $\mu$ in $\mathcal{M}$ if and only if $\mu$ is a RPC with parameter $x$ and $x_n$ converges to $x$ in $L^1([0,1])$. In particular, the RPC's form a compact subset of $\mathcal{M}$. 
\label{prop:rpc}
\end{prop}
\begin{proof}
For right-continuous increasing functions on a compact interval, the $L^1$ convergence is equivalent to the pointwise convergence of the right-continuous inverse. Therefore if $x_n\to x$ in $L^1$ it follows from \eqref{eqn:rpc} that $\mu_{x_n}$ converges to $\mu_x$ in $\mathcal{M}$ and the $\Leftarrow$ part of the claim is proven. Now suppose that $\mu_{x_n}$ converges to $\mu$ in $\mathcal{M}$. Denote the respective random overlap matrices by $Q_n$ and $Q$. By the Skorohod's representation theorem (see e.g. Corollary 6.12 in \cite{Kallenberg}), there exists a probability space such that $Q_n\to Q$ almost surely. By \eqref{eqn:rpc}, it follows that for every pair $i,j$, $x_n^{-1}(e^{-\tau_{ij}})$ converges to some number $y_{ij}$. On the other hand, the collection of stopping time $\{\tau_{ij}\}_{i<j}$ is dense in $(0,\infty)$ $\Gamma$-a.s. This is a basic consequence of the form of the generator of $\Gamma$ c.f. Proposition 1.3 in \cite{BS}. We conclude that the collection of limits $y_{ij}$ defines a right-continuous increasing function that we call $y$. By construction $x_n^{-1}\to y$ pointwise so that $x_n\to y^{-1}$ in $L^1$. The $\Longrightarrow$ part follows by taking $x:=y^{-1}$.

\end{proof}
\section{Properties of the limiting Gibbs Measure}
\label{section gg}
It would be desirable to characterize the Gibbs measure of spin glasses simply in terms of its infinite-volume properties. The Parisi theory for the SK model predicts that the infinite-volume Gibbs measure ought to be a Ruelle Probability Cascade, as far as observables of the form \eqref{eqn:h} are concerned. An appealing strategy to single out the RPC's is to characterize elements in $\mathcal{M}$ that are stochastically stable, a property exhibited by a large class of spin glasses including the SK model \cite{AC,Contucci,Guerra_cavity}. In the first part of this section, we discuss how the stochastic stability of the Gibbs measure of a spin glass is expressed by the infinite-volume measure. Such measures turn out to be natural stochastically-invariant measure on $\hilbert$ and are believed to characterize the RPC's, see e.g. \cite{argaiz}. Second, in the spirit of Parisi and Talagrand \cite{Parisi}, we show that the directing measure, as a measure on $\hilbert$, must be singular whenever they are satisfied.

\subsection{Stochastic Stability and the Ghirlanda-Guerra Identities}
We define the stochastic stability property as in \cite{AC, Contucci}. Consider a measure $\mu_{\beta,N}$ of the form \eqref{eqn:gibbs}. We denote by $\mu_{\beta,\lambda,N}$ the perturbed measure where $H_N(\sigma)$ is replaced by the Hamiltonian $\tilde{H}_N(\sigma):=H_N(\sigma) + \lambda K(\sigma)$ where $K$ is a Gaussian field on $\{-1,+1\}^N$ independent of $H_N$ with covariance matrix $\E[K(\sigma)K(\sigma')]=\frac{1}{N}\E[H_N(\sigma)H_N(\sigma')]$. 
\begin{df}
The sequence of measure $\{\mu_{\beta,N}\}$ is said to be {\it stochastically stable} at $\beta>0$ and $\lambda>0$
if for any bounded continuous function $F_s(\{q_{ij}\})$ on $s$ replicas
$$ \lim_{N\to\infty}\overline{\mu_{\beta,\lambda,N}^{(s)}}(F_s(\{q_{ij}\})=\lim_{N\to\infty}\overline{\mu_{\beta,N}^{(s)}}(F_s(\{q_{ij}\}) \ .$$
\end{df}
It was proven by Contucci and Giardina that a large class of spin glasses, including the SK model, are stochastically stable for almost all $\beta$ \cite{Contucci}. We remark that the stochastic stability translates into a natural invariance property of the limiting Gibbs measure. 

Let $\kappa=(\kappa(v),\|v\|_{\hilbert}\leq 1)$ be a Gaussian field labeled by the unit ball of $\hilbert$ with covariance given by a function of the inner product on $\hilbert$ - in such a way that the quadratic form is positive definite. We define the mapping 
\begin{eqnarray}
\Phi_{\lambda}: \mathcal{M}& \to &\mathcal{M} \\
\mu & \mapsto &\frac{\mu(dv)\ e^{\lambda \kappa(v)}}{\int_{\hilbert}\mu(dv')\ e^{\lambda \kappa(v')}}
\end{eqnarray}
where the randomness of the image $\Phi_\lambda\mu$ is induced by the randomness of $\mu$ and the one of $\kappa$.
It can be checked that the function $v\mapsto \kappa(v)$ is measurable $\kappa$-a.s. Moreover, the normalization is finite a.s. since by Fubini
\be
\E_{\kappa}\left[\int_{\hilbert}\mu(dv')\ e^{\lambda \kappa(v')}\right]\leq e^{\lambda^2/2}\ .
\label{eqn:mu mapping}
\ee
Finally, since the law of $\kappa$ is invariant under isometry, the image is an element of $\mathcal{M}$.  The above remarks ensures that the mapping $\Phi_\lambda$ is well-defined.

Again since observables of the form \eqref{eqn:h} determine the directing measure in $\mathcal{M}$ and since this space is compact, the following directly holds.
\begin{prop}
Let $\{\mu_{\beta,N}\}$ be stochastically stable. There exists a
subsequence such that $\{\Phi_{\lambda}\mu_{\beta,N_k}\}$ and $\{\mu_{\beta,N_k}\}$ converge respectively to $\tilde{\mu}_{\beta}$, $\mu_\beta$ and for which
$$ \tilde{\mu}_\beta\distrib\mu_\beta\ .$$
\end{prop}
It would be interesting to show whether or not the mapping $\Phi_\lambda$ is continuous in the $\mathcal{M}$ topology. If so, limit points of stochastically stable measures would be fixed points of $\Phi_\lambda$. 

This mapping is a natural transformation of a probability measure on $\hilbert$ where the weight associated with each vector is scaled by a function of the Gaussian variable of this point, when $\hilbert$ is seen as a Gaussian Hilbert space. It is the continuous generalization of the correlated evolution for competing particle systems or ROSt's as considered in \cite{RA, argaiz} when the directing measure is discrete. In that case, the weight of a vector is the position of a particle and the isometry is replaced by the ordering of the positions. The mapping $\Phi_\lambda$ corresponds to the incrementation of the positions of the particles correlated through the overlap matrix. It seems to be a fundamental issue to characterize invariant measure under \eqref{eqn:mu mapping}. In the discrete setting, it was shown under some robustness conditions that the Ruelle Probability Cascades are the only ones \cite{argaiz}, thereby singling out the ultrametric structure of the directing measure.

Deep results can already be derived from the stochastic stability property. One of the most studied is the fact that stochastic stability implies non-trivial relations between observables known as the Ghirlanda-Guerra identities and the Aizenman-Contucci polynomials \cite{GG,AC, Contucci, lp_gg}. In this paper, we focus on the strongest form of these identities.
\begin{df}
A sequence of random probability measures $\{\mu_N\}$ is said to satisfy the Ghirlanda-Guerra identities if
\begin{multline}
\overline{\mu_N^{(s+1)}}(F(q_{1,s+1}; \{q_{ij}\}_{i,j<s})=\\ \frac{1}{s}\overline{\mu_N}\times\overline{\mu_N^{(s)}}(F(q_{1,s+1}; \{q_{ij}\}_{i,j<s})+\frac{1}{s}\sum_{j=2}^{s} \overline{\mu_N^{(s)}}(F(q_{1j}; \{q_{ij}\}_{i,j<s}) + o(1)
\label{eqn gg1}
\end{multline}
for all $s\in\N$ and all bounded continuous functions $F$ that depend on the overlaps of $s$ sampled vectors $\{q_{ij}\}_{i,j<s}$ and of the overlap of an $s+1$-th vector with the first sampled one $q_{1,s+1}$.  
\end{df}
Ghirlanda and Guerra proved that the identities hold for the Gibbs measure of the SK model at almost all $\beta>0$. Since the functionals appearing in the identities are continuous functions in $\mathcal{M}$, it follows that the identities are fulfilled by any limiting Gibbs measure.
\begin{prop}
If $\{\mu_N\}\subset \mathcal{M}$ satisfies \eqref{eqn gg1}, then any of its limit points $\mu$ must obey
\begin{multline}
\overline{\mu^{(s+1)}}(F(q_{1,s+1}; \{q_{ij}\}_{i<j})=\\ \frac{1}{s}\overline{\mu}\times\overline{\mu^{(s)}}(F(q_{1,s+1}; \{q_{ij}\}_{i<j})+\frac{1}{s}\sum_{j=2}^{s} \overline{\mu^{(s)}}(F(q_{js}; \{q_{ij}\}_{i<j})
\label{eqn gg}
\end{multline}
for all $s\in\N$ and all bounded continuous function $F$ as above.
\end{prop}
The interest of the last observation is that one can now study the Ghirlanda-Guerra identities as an identity on the sampling probability measure on $\hilbert$ as opposed to relations on infinite overlap matrices.

\subsection{Singularity of the Gibbs Measure}
What does it mean for a probability measure on $\hilbert$ to satisfy the Ghirlanda-Guerra identities ? Parisi and Talagrand proved that the distribution on $[-1,1]$ induced by the overlap of two sampled vectors must be discrete \cite{Parisi}. We apply the same idea to investigate the singularity of the full measure in $\mathcal{M}$. We remark that the Ruelle Probability Cascades are known to satisfy the Ghirlanda-Guerra identities \cite{Bovier, lp_gg}. The directing measure of a RPC being pure point, one could conjecture that this is always the case whenever the identities are fulfilled. We prove a weaker version of this assertion.

\begin{lem}
Consider the event $\mathcal{A}_s:=\{q_{1s}\neq q_{1i} \ \forall \ 2\leq i\leq s-1\}$ that the inner product of the $s$-th sampled vector with the first sampled one differs from all the previous ones. If $\mu$ satisfies the identities \eqref{eqn gg}, then for any realization of $\mu$ a.s.
$$ \liminf_{s\to \infty} \mu^{(s)}(\mathcal{A}_s)= 0 \ .$$
\label{lem:gg}
\end{lem}
\begin{proof}
Following \cite{Parisi}, for $\delta>0$ we define the function
$$ F_\delta(q_{1s}; \{q_{1i}\}_{i<s-1})=1-\min\{1, \frac{1}{\delta} |q_{1s}-q_{1i}| \ i=2,...,s-1\} \ .$$
This a bounded continuous function of the entries. Moreover, it is easily checked that as $\delta\to 0$ $F_\delta$ converges pointwise to the indicator function $\chi_{\mathcal{A}_s^c}$ of $\mathcal{A}_s^c=\{\exists \ 2\leq i\leq s-1 : q_{1s}= q_{1i} \}$. Plainly, $\chi_{\mathcal{A}^c_s}(q_{1i}, \{q_{1i}\}_{i<s-1})\equiv1$.

Plugging $F_\delta$ into \eqref{eqn gg} and using dominated convergence as $\delta\to 0$, we get
$$ \overline{\mu^{(s)}}(\mathcal{A}_s^c)=\\ \frac{1}{s-1}\ \overline{\mu}\times\overline{\mu^{(s-1)}}(\mathcal{A}_s^c)+\frac{s-2}{s-1}.$$
Hence $\lim_{s\to\infty }\overline{\mu^{(s)}}(\mathcal{A}_s)=0$ a.s. Moreover by Fatou's Lemma
$$0=\lim_{s\to\infty}\overline{\mu^{(s)}(\mathcal{A}_s)}\geq \overline{\liminf_{s\to\infty }\mu^{(s)}(\mathcal{A}_s)}\geq 0 $$
and the desired conclusion holds.
\end{proof}

\begin{cor}
If $\mu\in\mathcal{M}$ satisfies the Ghirlanda-Guerra identities \eqref{eqn gg}, then it must be singular in the following sense:
Let $B\subset{\R^d}$ be a closed ball in $\R^d$ for any $d\in\N$. If $\mu(B)>0$, the conditional probability $\mu(\ |B)$ is singular a.s. with respect to the Lebesgue measure on $B$.
\label{cor:gg} 
\end{cor}
\begin{proof}
Let $\mathcal{B}_{s}=\{v_i\in B\ \forall i=1,...,s\}$ be the event that $s$ vectors are sampled in $B$. Note that $\mu^{(s)}(\mathcal{B}_{s})>0$ whenever $\mu(B)>0$ so the conditional measure $\mu^{(s)}(\ |\mathcal{B}_{s})$ is well-defined.  It immediately follows from Lemma \ref{lem:gg} that a.s.
\be
\liminf_{s\to \infty} \mu^{(s)}(\mathcal{A}_s\ |\mathcal{B}_{s})= 0 \ .
\label{eqn:mucond}
\ee
Now suppose that with positive probability $\mu(\ |B)$ is absolutely continuous with respect to the Lebesgue measure on $B$. Then so is $\mu^{(s)}(\ \ |\mathcal{B}_{s})$ for the Lebesgue measure $\nu_B^{(s)}$ on $B^s$. In particular the event $\mathcal{A}_s^c=\{\exists \ 2\leq i\leq s-1 : q_{1s}= q_{1i} \}$ has probability $\mu^{(s)}(\mathcal{A}_s^c\ |\mathcal{B}_{s})=0$ for all $s$ since $\nu_{B}^{(s)}(\mathcal{A}_s^c)=0$. We conclude that $\lim_{s\to\infty} \mu^{(s)}(\mathcal{A}_s\ |\mathcal{B}_{s})=1$ with positive probability thereby contradicting \eqref{eqn:mucond}.
\end{proof}

\end{document}